\documentclass[12pt, leqno]{amsart}
\pdfoutput=1
\usepackage{amsmath, amssymb, amsthm}
\usepackage{rsfso}
\usepackage[svgnames]{xcolor}
\usepackage{tikz}
\usetikzlibrary{arrows,cd}
\tikzset{>=latex}
\usepackage{mathtools}
\usepackage[alwaysadjust]{paralist}
\usepackage[alphabetic]{amsrefs}
\usepackage[T1]{fontenc}
\usepackage{mathptmx}
\usepackage{microtype}
\usepackage[colorlinks = true,
            linkcolor  = DarkBlue,
            urlcolor   = DarkRed,
            citecolor  = DarkGreen]{hyperref}
\hypersetup{
  pdftitle={Discriminant Loci},
  pdfauthor={Hirotachi Abo, Robert Lazarsfeld, and Gregory G. Smith}
  }
\usepackage[centering, includeheadfoot, hmargin=1.0in,%
  tmargin=1.0in, bmargin=1in, headheight=6pt]{geometry}
\newtheorem{lemma}{Lemma}[section]
\newtheorem{theorem}[lemma]{Theorem}
\newtheorem{corollary}[lemma]{Corollary}
\newtheorem{proposition}[lemma]{Proposition}
\theoremstyle{definition}

\newtheorem{remark}[lemma]{Remark}
\newtheorem{examplex}[lemma]{Example}
\newenvironment{example}
  {\pushQED{\qed}\examplex}
  {\popQED\endexamplex}

\newcommand{\binomial}[2]{\Bigl(%
  \genfrac{}{}{0pt}{}{\raisebox{-1pt}{$#1$}}{\raisebox{2pt}{$#2$}} \Bigr)}

\newcommand{\CC}{\ensuremath{\mathbb{C}}}

\newcommand{\PP}{\ensuremath{\mathbb{P}}}
\newcommand{\ZZ}{\ensuremath{\mathbb{Z}}}
\newcommand{\OO}{\ensuremath{\mathcal{O}}}
\DeclareMathOperator{\codim}{codim}
\DeclareMathOperator{\defect}{def}
\DeclareMathOperator{\rank}{rank}
\DeclareMathOperator{\Sym}{Sym}
\begin{document}
\title[Discriminant loci]%
{Ramification and Discriminants of Vector Bundles \\
  and a Quick Proof of Bogomolov's Theorem}

\author[H.~Abo]{Hirotachi Abo}
\address{Hirotachi Abo: Department of Mathematics, University of Idaho,
  Moscow, Idaho 83844--1103, United States of America;
  {\normalfont \texttt{abo@uidaho.edu}}}
\author[R.~Lazarsfeld]{Robert Lazarsfeld}
\address{Robert Lazarsfeld: Department of Mathematics, Stony Brook University,
  Stony Brook, New York 11794, United States of America;
  {\normalfont \texttt{robert.lazarsfeld@stonybrook.edu}}}
\author[G.G.~Smith]{Gregory G. Smith}
\address{Gregory G.{} Smith: Department of Mathematics and Statistics, Queen's
  University, Kingston, Ontario, K7L 3N6, Canada;
  {\normalfont \texttt{ggsmith@mast.queensu.ca}}} 

\thanks{2020 \emph{Mathematics Subject Classification}. 14J60; 14C17, 14N05}

\dedicatory{Dedicated to Giorgio Ottaviani on the occasion of his sixtieth
  birthday.}

\begin{abstract}
  By analyzing degeneracy loci over projectivized vector bundles, we
  recompute the degree of the discriminant locus of a vector bundle and
  provide a new proof of the Bogomolov instability theorem.
\end{abstract}

\maketitle

\addtocounter{section}{-1}
\section{Introduction}

\noindent
Let $X$ be an $n$-dimensional smooth complex projective variety and let $E$ be
a globally generated vector bundle on $X$ of rank $e \leqslant n$. The
projective space $\PP^r = \PP \bigl( H^{0}(X, E)^* \!\bigr)$ parameterizes
sections of $E$ up to scalars. The \emph{discriminant} of $E$ is the locus in
$\PP^r$, typically a hypersurface, defined by
\[
  \Delta(E) \coloneqq \bigl\{ s \in \PP^r \mathrel{\big|} \text{the zero scheme
    $\operatorname{Zeroes}(s)$ of $s$ is singular} \big\} \, .
\]
The closed algebraic set $\operatorname{Zeroes}(s)$ is understood to have its
natural scheme structure: when $e = n$, $\Delta(E)$ consists of those sections
that vanish at something other than $\int c_n(E)$ distinct points.  There are
various situations where it is of interest to calculate the degree of
$\Delta(E)$.  This comes up, for instance, in connection with eigenvalues of
tensors \cite{ASS}. In \cite{A}, the first author derives a formula for the
degree when $e = n$ and $X = \PP^n$.

The first purpose of this note is to give a very quick derivation of a formula
for the (virtual) degree of $\Delta(E)$ reproving some results from \cite{LM}.
For example, when $e = n$, we show that the expected degree of $\Delta(E)$ is
given by
\[
  \delta(E)
  = \int_X  \bigl( K_X + c_1(E) \!\bigr) \; c_{n-1}(E)  + n \, c_n(E)  \, .
\]
If each section $s$ in $\Delta(E)$ is singular at several points, then the
actual degree of the discriminant hypersurface is smaller than its postulated
one. However, when $E$ is very ample and $1$-jet spanned, we also show that
$\Delta(E)$ is irreducible of the expected degree.

As one might expect, the basic idea is to compute the class of the singular
locus of the universal zero-locus over $\PP^r$. It turns out that a somewhat
related computation leads to an extremely quick proof of the Bogomolov
instability theorem for vector bundles of rank $2$ on an algebraic surface,
reducing the statement in effect to the Riemann--Hurwitz formula. The
existence of a proof along these lines seems to have been known to the
experts, but as far as we can tell it is not generally familiar. We therefore
take this occasion to present the argument.  Some time ago, Langer
\cite{Langer}*{Appendix} gave an even quicker, but related proof, using the
fact that stability is preserved under pulling back by generically finite
morphisms.

The formula for the ramification locus is derived in Section~1. In Section~2,
we show that, when $E$ is very ample and $1$-jet spanned, the discriminant
locus is irreducible of the expected degree. The proof of the Bogomolov
instability theorem occupies Section~3.

\subsection*{Conventions}
We work throughout over the complex numbers $\CC$.  For any vector space $V$
or vector bundle $E$, $\PP(V)$ or $\PP(E)$ denotes the projective space of
one-dimensional quotients.  Given a smooth variety $X$, the Chow ring of $X$
is $A^{\bullet}(X)$ (or, if the reader prefers, this is the even cohomology
ring $H^{2\bullet}(X, \ZZ)$). We write $c_i(E)$ and $s_i(E)$ for the $i$-th
Chern and Segre classes of a vector bundle $E$ whereas $c(E)$ and $s(E)$ are
the corresponding total Chern and Segre classes. Following the convention of
\cite{Fulton}*{Example~3.2.7}, we use the notation
$c(E - F) \coloneqq c(E)/c(F) = c(E) \; s(F)$ for the ``difference'' of the
total Chern classes of two bundles. Finally, given a class $\alpha$ in
$A^\bullet(X)$, the component of $\alpha$ in codimension $k$ is
$\alpha_k \in A^k(X)$.

\section{Ramification Locus}

\noindent
In this section, we derive a formula for ramification class of certain
morphisms from projectivized vector bundles.  To be more explicit, fix an
$n$-dimensional smooth complex projective variety $X$ and consider a
globally-generated vector bundle $E$ on $X$ of rank $e$ such that
$e \leqslant n$.

Let $V_E \coloneqq H^0(X, E)$ be the vector space of global sections of $E$
and set $r \coloneqq \dim_\CC V_E - 1$.  The trivial vector bundle on $X$ with
fibre $V_E$ is denoted $V_E \otimes_{\CC} \OO_X$ and the kernel of the
evaluation map $\operatorname{ev}_E \colon V_E \otimes_{\CC} \OO_X \to E$ is
$M_E \coloneqq \operatorname{Ker}( \operatorname{ev}_E)$.  It follows that
$M_E$ is a vector bundle of rank $r-e+1$ sitting in the short exact sequence
\[
  \begin{tikzcd}
    0 \arrow[r]
    & M_E \arrow[r]
    & V_E \otimes_{\CC} \OO_X \arrow[r, "\operatorname{ev}_E"]
    & E \arrow[r]
    & 0 \, . 
  \end{tikzcd}
\]
Applying the duality functor $(-)^* \coloneqq \mathcal{H}\!\textit{om}(-,
\OO_{X})$, we obtain the short exact sequence
\[
  \begin{tikzcd}
    0 \arrow[r] & E^* \arrow[r] & (V_E \otimes_{\CC} \OO_X \!)^* \arrow[r]
    & M_E^* \arrow[r] & 0 \, . 
  \end{tikzcd}
\]
The surjective map onto $M_E^*$ identifies the projectivization $\PP(M_E^*)$
with a closed subscheme in the product
$\PP\bigl(\! (V_E \otimes_{\CC} \OO_X \!)^* \!\bigr)\! = X \times \PP(V_E^*)$
where $V_E^{*}$ is the dual vector space of $V_E$.  Specifically, we have
$\PP(M_E^*) = \{ (x,[s]) \in X \times \PP(V_E^*) \mathrel{|} \text{$s(x) = 0$}
\}$.  Let $p_E \colon X \times \PP(V_E^*) \to X$ be the projection onto the
first factor.  For notational simplicity, we also use $p_E$ for the
restriction to $\PP(M_E^*)$.  Let $q_E \colon \PP(M_E^*) \to \PP(V_E^*)$ be
the restriction of the projection from $X \times \PP(V_E^*)$ onto the second
factor $\PP(V_E^*)$.  When the vector bundle $E$ is unnecessary, we omit the
subscripts on $V$, $M$, $p$, and $q$.

Guided by Example~14.4.8 in \cite{Fulton}, the \emph{ramification locus}
$R(q)$ of the map $q \colon \PP(M^*) \to \PP(V^*)$ is the $(r-1)$-st
degeneracy locus of the induced differential
$dq \colon \smash{q^* \Omega_{\PP(V^*)} \to \Omega_{\PP(M^*)}}$;
\[
  R(q) \coloneqq \bigl\{ x \in \PP(M^*) \mathrel{\big|}
  \text{rank of map $dq$ at the point $x$ is at most $r-1$} \bigr\}
  = \operatorname{Zeroes}(\textstyle\bigwedge\nolimits^{\!r} dq)
  \, .
\]
Since $\PP(V^*)$ and $\PP(M^*)$ have dimension $r$ and $n+r-e$, the
determinantal subscheme $R(q)$ has codimension at most
$\bigl( r - (r-1) \!\bigr) \bigl(n+r-e - (r-1) \!\bigr) = n-e+1$; see
\cite{Fulton}*{p.~242}.  The next proposition provides a formula for the
\emph{ramification class} $[R(q)]$ in the Chow ring
$A^{\bullet} \!\bigl( \PP(M^*) \!\bigr)$.

\begin{proposition}
  \label{p:deg}
  When the ramification locus $R(q)$ has codimension $n-e+1$, its class in
  $A^{\bullet} \!\bigl( \PP(M^*) \!\bigr)$ is
  $[R(q)] = \bigl\{\! c (p^* \Omega_X) \; s \bigl(p^* E^* \otimes
  \OO_{\PP(M^*)}(-1) \!\bigr) \!\bigr\}_{\!n-e+1}$ and the degree of its
  pushforward is
  \[
    \deg {q}_*[R(q)]
    = \int_X  p_*  \!\Bigl(\! [R(q)] \; \smash{c_1 \!\bigl( \OO_{\PP(M^*)}(1)
    \!\bigr)}^{r-1} \!\Bigr) \, .
  \]
\end{proposition}

\begin{proof}
  Since $R(q)$ has codimension $n-e+1$, the Thom--Porteous
  formula~\cite{Fulton}*{Theorem 14.4} establishes that
  $[R(q)] = c_{n-e+1} \bigl( \Omega_{\PP(M^*)}- q^* \Omega_{\PP(V^*)}
  \bigr)$.  Hence, it suffices to prove that
  \[
    c_{n-e+1} \bigl( \Omega_{\PP(M^*)}- q^* \Omega_{\PP(V^*)} \bigr)
    = c_{n-e+1} \!\bigl( p^* \Omega_X- p^* E^* \otimes
    \OO_{\PP(M^*)}(-1) \!\bigr)
    \, .
  \]
  By combining the two short exact sequences
  \[
    \begin{tikzcd}[row sep = 0.0em]
      0 \ar[r]
      & {\mathcal{I}^{}_{\PP(M^*)} / \mathcal{I}^{2}_{\PP(M^*)}} \ar[r, "\delta"]
      & \Omega_{X \times \PP(V^*)} \big|_{\PP(M^*)} \arrow[r]
      & \Omega_{\PP(M^*)} \ar[r]
      & 0 \, \phantom{,}\\
      0 \arrow[r]
      & q^* \Omega_{\PP(V^*)} \ar[r]
      & \Omega_{X \times \PP(V^*)} \big|_{\PP(M^*)} \ar[r, "\theta"]
      & p^* \Omega_X \ar[r]
      & 0 \, , 
    \end{tikzcd}
  \]
  we obtain the commutative diagram:
  \[
    \begin{tikzcd}[row sep = 1.75em]
      & & 0 \arrow[d] & & \\[-2pt]
      & & q^* \Omega_{\PP(V^*)} \ar[d] \ar[dr, "dq"] & & & \\[-2pt]
      0 \arrow[r]
      & \mathcal{I}^{}_{\PP(M^*)} / \mathcal{I}^{2}_{\PP(M^*)} \ar[r, "\delta"] \arrow[dr]
      & \Omega_{X \times \PP(V^*)} \big|_{\PP(M^*)} \ar[r] \ar[d, "\theta"]
      & \Omega_{\PP(M^*)} \arrow[r] & 0 \, .  \\[-2pt]
      & & p^* \Omega_X \arrow[d] & & \\[-2pt]
      & & 0 & &
    \end{tikzcd}
  \]
  The snake lemma shows that
  $\operatorname{Coker}(dq) \cong \operatorname{Coker}(\theta \circ
  \delta)$, so we deduce that
  \[
    c_{n-e+1} \!\bigl( \Omega_{\PP(M^*)} - q^* \Omega_{\PP(V^*)} \bigr)
    = c_{n-e+1} \!\bigl( p^* \Omega_{X} - \mathcal{I}^{}_{\PP(M^*)} /
    \mathcal{I}^{2}_{\PP(M^*)} \bigr) \, .
  \]
  
  To prove the first part, it remains to show that the conormal bundle
  $\smash{\mathcal{I}^{}_{\PP(M^*)} / \mathcal{I}^{2}_{\PP(M^*)}}$ on
  $\PP(M^*)$ is isomorphic to the vector bundle
  $p^* E^* \otimes \OO_{\PP(M^*)}(-1)$.  As a closed subscheme of
  $X \times \PP(V^*)$, the projectivization $\PP(M^*)$ is the zero scheme of a
  regular section of $p^* E \otimes \OO_{X \times \PP(V^*)}(1)$; see
  \cite{Fulton}*{Appendix~B.5.6}.  Tensoring the Koszul complex associated to
  this regular section with $\OO_{\PP(M^*)}$ produces the desired isomorphism
  $\smash{p^* E^* \otimes \OO_{\PP(M^*)}(-1) \cong \mathcal{I}_{\PP(M^*)}
    \otimes \OO_{\PP(M^*)} \cong \mathcal{I}^{}_{\PP(M^*)} /
    \mathcal{I}^{2}_{\PP(M^*)}}$.

  For the second part, observe that
  $\OO_{\PP(M^*)}(1)= q^* \OO_{\PP(V^*)}(1)$; see
  \cite{Lazarsfeld}*{Example~6.1.5}. It follows from the projection formula
  that the degree of pushforward is
  \begin{align*}
    \deg q_* [R(q)]
    &= \int_{\PP(V^*)} q_* [R(q)] \;
      \smash{c_1 \!\bigl( \OO_{\PP(V^*)}(1) \!\bigr)}^{r-1} 
      = \int_{\PP(M^*)}  q^* \!\Bigl(\! q_*[R(q)] \;
      \smash{c_1 \!\bigl( \OO_{\PP(V^*)}(1) \!\bigr)}^{r-1} \!\Bigr) \\ 
    &= \int_{\PP(M^*)} [R(q)] \;
      \smash{ c_1 \!\bigl(\OO_{\PP(M^*)}(1) \!\bigr)}^{r-1} 
      = \int_X p_* \!\Bigl( [R(q)] \;
      \smash{c_1 \!\bigl( \OO_{\PP(M^*)}(1) \!\bigr)}^{r-1} \!\Bigr)
      \, .  \qedhere
  \end{align*}
\end{proof} 

In the following examples, we examine three special cases expressing
ramification class as a polynomial in the Chern classes for $E$ and
$\Omega_{X}$.  From the defining short exact sequence of the kernel bundle
$M_E$, we see that
$p_* \, \smash{c_1 \!\bigl( \OO_{\PP(M^*)}(1) \!\bigr)}^{r-e+i} = s_i(M) =
c_i(E)$ for all nonnegative integers $i$.

\begin{example}[$e=1$]
  \label{e:e=1}
  Suppose that the vector bundle $E$ has rank $1$.  When ramification locus
  $R(q)$ has codimension $n$, Proposition~\ref{p:deg} implies that
  \begin{align*}
    [R(q)]
    &= \bigl\{\! c(p^* \Omega_X) \; s \bigl( p^*E^* \otimes
      \OO_{\PP(M^*)}(-1) \!\bigr) \!\bigr\}_{\!n} 
      = \smash{\sum_{i=0}^n} \, c_{n-i} (p^* \Omega_X) \; (-1)^i
      \smash{c_1 \!\bigl( p^* E^* \otimes \OO_{\PP(M^*)}(-1) \!\bigr)}^{i}
    \\
    &= \sum_{i=0}^n c_{n-i}(p^* \Omega_X) \;
      \sum_{j=0}^i \binomial{i}{j}  c_1( p^* E)^j \,
      \smash{c_1 \!\bigl( \OO_{\PP(M^*)}(1) \!\bigr)}^{i-j} \, ,
  \end{align*}
  and
  $\deg q_*[R(q)] = \displaystyle\sum_{i=0}^n (i+1) \int_X c_{n-i}(\Omega_X) \;
  c_1(E)^i$. \qedhere
\end{example}

\begin{example}[$n = e$]
  \label{e:n=e} 
  Suppose that the rank of the vector bundle $E$ equals the dimension of its
  underlying variety $X$.  When $R(q)$ has codimension $1$,
  Proposition~\ref{p:deg} implies that
  \begin{align*}
    [R(q)]
    &= \bigl\{\! c (p^* \Omega_X) \; s \bigl( p^* E^* \otimes
      \OO_{\PP(M^*)}(-1) \!\bigr) \!\bigr\}_{\!1}  
      = c_1 ( p^* \Omega_X)
      - c_1 \!\bigl( p^* E^* \otimes \OO_{\PP(M^*)}(-1) \!\bigr) \\[-2pt]
    &= c_1( p^* \Omega_X) + c_1(p^* E)
      + n \, c_1 \bigl( \OO_{\PP(M^*)}(1) \!\bigr)   
  \end{align*}
  and
  $\deg q_*[R(q)] = \displaystyle\int_X \bigl( c_1(\Omega_X) + c_1(E)
  \!\bigr) \, c_{n-1}(E) + n \, c_n(E)$. \qedhere
\end{example}
  
\begin{example}[$e = n-1$]
  Suppose that the rank of $E$ is the dimension of $X$ minus $1$.  Observe
  that
  \begin{align*}
    s_2(p^*E)
    &= s_1(p^* E)^2-c_2(p^* E^*) = c_1(p^* E)^2 - c_2(p^* E) \, , \\
    s_2 \!\bigl( p^* E^* \otimes \OO_{\PP(M^*)}(-1) \!\bigr)
    & = \binomial{n}{n-2} \smash{c_1 \!\bigl( \OO_{\PP(M^*)}(1) \!\bigr)}^2
      -n \, s_1(p^*E^*) \; \smash{c_1 \!\bigl( \OO_{\PP(M^*)}(1) \!\bigr)}
      + s_2(p^*E^*) \\[-2pt]
    &= \binomial{n}{2} \smash{c_1 \!\bigl( \OO_{\PP(M^*)}(1) \!\bigr)}^2
      - n \, c_1(p^* E) \; \smash{c_1 \!\bigl( \OO_{\PP(M^*)}(1) \!\bigr)}
      + c_1(p^* E)^2 - c_2(p^* E) \, ;
  \end{align*}
  see \cite{Fulton}*{p.~50 and Example~3.1.1}.  When $R(q)$ codimension $2$,
  Proposition~\ref{p:deg} implies that
  \begin{align*}
    [R(q)]
    &= \bigl\{ c(p^* \Omega_X) \;
      s \bigl( p^* E^* \otimes \OO_{\PP(M^*)}(-1) \!\bigr) \!\bigr\}_{\!2} \\
    &= c_2(p^* \Omega_X) + c_1(p^* \Omega_X) \;
      s_1 \!\bigl( p^* E^*  \otimes \OO_{\PP(M^*)}(-1) \!\bigr)
      + s_2 \bigl( p^* E^* \otimes \OO_{\PP(M^*)}(-1) \!\bigr) \\
    &= c_2(p^* \Omega_X) + c_1(p^* \Omega_X) \;
      \!\Bigl(\! c_1(p^* E)+(n-1) \; \smash{c_1 \!\bigl( \OO_{\PP(M^*)}(1)
      \!\bigr)} \!\Bigr) \\[-2pt]
    &\phantom{WWWW} + \binomial{n}{2}
      \smash{c_1 \!\bigl( \OO_{\PP(M^*)}(1) \!\bigr)}^2
      - n \, c_1(p^*E) \; \smash{c_1 \!\bigl( \OO_{\PP(M^*)}(1) \!\bigr)}
      + c_1(p^* E)^2 - c_2(p^* E) 
  \intertext{and}
    \deg q_*[R(q)]
    &=
      \int_X \bigl( c_2(\Omega_X)
      + c_1(\Omega_X) \; c_1(E) + c_1(E)^2 - c_2(E) \!\bigr) \; c_{n-2}(E) \\[-2pt]
    &\phantom{WWWW} + \bigl(\! (n-1) \, c_1(\Omega_X) + n \, c_1(E) \!\bigr)
      \; c_{n-1}(E)  \, . \qedhere
  \end{align*}
\end{example}

\section{Discriminant Locus of a Vector Bundle}

\noindent
This section determines the degree of the discriminant of a vector bundle.  As
in the first section, $X$ is an $n$-dimensional smooth complex projective
variety $X$ and $E$ is a globally-generated vector bundle on $X$ of rank
$e \leqslant n$.  Set $V_E \coloneqq H^{0}(X, E)$, let $M_E$ be the kernel of
the evaluation map
$\operatorname{ev}_{E} \colon V_E \otimes_{\CC} \OO_X \to E$, and write
$q_E \colon \PP(M_{E}^{*}) \to \PP(V_E^*)$ for composition of the canonical
inclusion $\PP(M_{E}^{*}) \to X \times \PP(V_E^*)$ and the projection
$X \times \PP(V_E^*) \to \PP(V_E^*)$ onto the second factor.

The \emph{discriminant locus} $\Delta(E)$ of the vector bundle $E$ is the
reduced scheme structure on the image of the ramification locus $R(q_E)$ under
the map $q_E$.  A section $s$ in $V_E^*$ is \emph{nonsingular} if its zero
scheme $\operatorname{Zeroes}(s)$ is nonsingular and has codimension $e$ in
$\PP(V_E^*)$; otherwise it is \emph{singular}.  With this terminology, one
verifies that
\[
  \Delta(E) \coloneqq \bigl\{ [s] \in \PP(V_{E}^{*}) \mathrel{\big|} \text{the
    section $s$ is singular} \bigr\} \, .
\]
The \emph{defect} of the vector bundle $E$ is the integer given by
$\defect(E) \coloneqq \codim \Delta(E) - 1$, the \emph{expected degree} of
$\Delta(E)$ is $\delta(E) \coloneqq \deg (q_E)_* [R(q_E)]$, and the
\emph{coefficient} of $R(q_E)$ in $[R(q_E) _{\mathrm{red}}]$ is the unique
positive integer $m_E$ such that $[R(q_E)] = m_E \, [R(q_E)_{\mathrm{red}}]$
in the Chow ring $A^{\bullet} \!\bigl( \PP(M_E^*) \!\bigr)$.

The significance of these numerical invariants becomes clear with an
additional hypothesis.

\begin{remark}
  \label{r:deg}
  Assume that the ramification locus $R(q_E)$ is irreducible and has dimension
  $r-1$ (or equivalently codimension $n-e+1$).  It follows that the
  discriminant locus $\Delta(E)$ is also irreducible.  For the function fields
  $\CC \bigl( R(q_E) \!\bigr)$ and $\CC \bigl( \Delta(E) \!\bigr)$ of the
  reduced schemes $R(q_E)_{\text{red}}$ and $\Delta(E)$, the degree of the
  field extension is
  $\bigl[ \CC \bigl( R(q_E) \!\bigr) \mathbin{:} \CC \bigl( \Delta(E) \!\bigr)
  \!\bigr]$ and the degree of $R(q_E)$ over $\Delta(E)$ is
  \[
    \deg R(q_E)/\Delta(E) \coloneqq
    \begin{cases}
      \bigl[ \CC \bigl( R(q_E) \!\bigr) \mathbin{:} \CC \bigl( \Delta(E)
      \!\bigr) \!\bigr]
      & \text{if $\dim \Delta(E) = r-1$} \\[2pt]
      0
      & \text{if $\dim \Delta(E) < r-1$.} 
    \end{cases}
  \]
  The definition of the pushforward of a cycle gives
  $(q_E)_* [R(q_E)] = m_E \, \bigl( \deg R(q_E)/\Delta(E) \!\bigr) \,
  [\Delta(E)]$; see \cite{Fulton}*{Section~1.4}.  Hence, we have
  $\defect(X) > 0$ if and only if $\delta(E) = 0$.  Moreover, when $R(q_E)$ is
  integral and birational to $\Delta(E)$, we have
  $\deg \Delta(E) = \delta(E)$.
\end{remark}

Although the next result is likely known to experts, we could not find an
adequate reference.

\begin{theorem}
  \label{t:ample}
  Assume that $X$ an $n$-dimensional smooth projective variety $X$ and let $E$
  be a very ample vector bundle on $X$ of rank $e \leqslant n$. Let
  $\pi \colon \PP(E) \to X$ be the projective bundle associated to $E$ and let
  $L \coloneqq \OO_{\PP(E)}(1)$ be the tautological line bundle on the
  projectivization $\PP(E)$.
  \begin{compactitem}
  \item The discriminant locus $\Delta(E)$ of the vector bundle $E$ is
    isomorphic to the discriminant locus $\Delta(L)$ of the line bundle $L$.
    In particular, $\Delta(E)$ is irreducible.
  \item When the discriminant locus $\Delta(E)$ is a hypersurface, the reduced
    scheme $R(q_E)_{\mathrm{red}}$ is birational to $\Delta(E)$ and
    $\deg \Delta(E) = m_E \, \bigl\{\! c (p^* \Omega_X) \; s \bigl(p^* E^*
    \otimes \OO_{\PP(M_E^*)}(-1) \!\bigr) \!\bigr\}_{\!n-e+1}$.
  \end{compactitem}
\end{theorem}

\begin{proof}
  The canonical isomorphism
  $\smash{V_E = H^0(X,E) \xrightarrow{\;\cong\;} H^0 \!\bigl( \PP(E), L \bigr)
    = V_L}$ induces an isomorphism
  $\varphi \colon \PP(V_L^*) \to \PP (V_E^*)$. It is enough to show that the
  restriction of $\varphi$ to $\Delta(L)$ yields an isomorphism from
  $\Delta(L)$ to $\Delta(E)$.  To accomplish this, it suffices to prove that a
  section $s$ in $V_E^*$ is singular if and only if the corresponding section
  $\tilde{s}$ in $V_L^*$ is singular.  As this assertion is local, we may
  assume that $X$ is affine and $E \cong \bigoplus_{i=1}^{e} \OO_X$. Hence,
  there exist $f_1, f_2, \dotsc, f_e \in H^0(X, \OO_X)$ such that
  $s = (f_1, f_2, \dotsc, f_e)$ and
  $\tilde{s} = f_1 \, x_1 + f_2 \, x_2 + \dotsb + f_e \, x_e$ where
  $x_1, x_2, \dotsc, x_e$ are homogeneous coordinates of
  $\PP^{e-1} = \PP(V_E^*)$.  The assertion now follows from a local
  calculation of derivatives as appears in~\cite{A}*{Subsection~3.2}.
  
  The same calculation shows that restriction of the map
  $\pi \times \varphi \colon \PP(E) \times \PP(V_L^*) \to X \times \PP(V_E^*)$
  to $R(q_L)_{\mathrm{red}}$ is a birational map from $R(q_L)_{\text{red}}$ to
  $R(q_E)_{\text{red}}$.  When $\Delta(L)$ is a hypersurface, Proposition~3.2
  in \cite{GKZ} demonstrates that reduced scheme $R(q_L)_{\text{red}}$ is
  birational to discriminant locus $\Delta(L)$.  It follows that the reduced
  scheme $R(q_E)_{\mathrm{red}}$ is birational to discriminant locus
  $\Delta(E)$. Finally, the degree formula is an immediate consequence of
  Remark~\ref{r:deg}.
\end{proof}

To prove that the ramification locus is reduced, we first record a general
observation about degeneracy loci.  Consider three vector bundles $A$, $B$,
and $C$ on a smooth projective variety $X$ together with an injective vector
bundle morphism $\mu \colon A \otimes B^* \to C$. Let
$\varpi \colon \PP(C) \to X$ be the projective bundle associated to $C$, let
$\eta \colon \varpi^* C \to \OO_{\PP(C)}(1)$ be the natural surjective
morphism, and let
$\tilde{\mu} \colon \varpi^* (A \otimes B^*) \to \OO_{\PP(C)}(1)$ be the
composition of $\mu$ with $\eta$.  Using tensor-hom adjunction, the map
$\tilde{\mu}$ corresponds to the morphism
$\mu' \colon \varpi^*A \to \varpi^* B \otimes \OO_{\PP(C)}(1)$.

\begin{lemma}
  \label{l:CM}
  For any nonnegative integer $k$, the $k$-th degeneracy locus
  $D_k(\mu') \coloneqq \operatorname{Zeroes}(\bigwedge^{k+1} \mu')$ is reduced
  and Cohen--Macaulay of codimension
  $\bigl( \rank(A) - k \bigr) \bigl(\rank(B) - k \bigr)$.
\end{lemma}

\begin{proof}
  As the assertion is local, we may assume that $X$ is affine and the three
  vector bundles are trivial. Let $U$, $V$, and $W$ be vector spaces such that
  $A = U \otimes_{\CC} \OO_X$, $B = V \otimes_{\CC} \OO_X$, and
  $C = W \otimes_{\CC} \OO_X$.  For each nonnegative integer $k$, let
  $D_k(U,V)$ be the locus of points in
  $\PP(U \otimes_{\CC} V^*) = \PP\bigl( \operatorname{Hom}_{\CC}(U,V)
  \!\bigr)$ whose corresponding linear transformations from $U$ to $V$ have
  rank at most $k$.

  Consider the projective bundle $\rho \colon \PP(A \otimes B^*) \to X$
  associated to $A \otimes B^*$. On the projectivization
  $Y \coloneqq \PP(A \otimes B^*)$, the surjective morphism
  $\theta \colon \rho^*(A \otimes B^*) \to \OO_Y(1)$ corresponds to the
  morphism $\theta' \colon \rho^* A \to \rho^*B \otimes \OO_Y(1)$ whose
  $k$-th degeneracy locus $D_k(\theta')$ is $X \times D_k(U,V)$.  In
  particular, $D_k(\theta')$ is reduced and Cohen--Macaulay of codimension
  $\bigl( \rank(A) - k \bigr) \bigl(\rank(B) - k \bigr)$.

  Let $Q$ be the cokernel of the map $\mu \colon A \otimes B^* \to C$.  It
  follows that $\PP(Q)$ is a subbundle of $\PP(C)$.  Let
  $\psi \colon \PP(C) - \PP(Q) \to Y$ be the associated trivial affine bundle
  over $X$.  Since the map
  $\mu' \colon \varpi^*A \to \varpi^* B \otimes \OO_{\PP(C)}(1)$ is nonzero
  away from $\PP(Q)$, we have the commutative diagram
  \[
    \begin{tikzcd}
      \PP(C) - \PP(Q) \ar[rr, "\psi"] \ar[dr, "\varpi\;\;\;" left] &
      & Y \ar[dl, "\rho"] \\
      & X & 
    \end{tikzcd}
  \]
  with the property that $\mu' = \psi^*(\theta')$.  Hence, the $k$-th
  degeneracy locus $D_k(\mu')$ is the ``cone'' over $D_k(\theta')$ in $\PP(C)$
  with vertex $\PP(Q)$; it is the product of $X$ and the cone over $D_k(U,V)$
  in $\PP(W)$ with vertex $\PP \bigl( W / (U \otimes V^*) \!\bigr)$.  We
  conclude that $D_k(\mu')$ is also reduced and Cohen--Macaulay of codimension
  $\bigl( \rank(A) - k \bigr) \bigl(\rank(B) - k \bigr)$.
\end{proof}

To ensure that the ramification locus $R(q_E)$ is reduced, we rely on a
stronger hypothesis than $E$ being very ample. To define this condition, we
use the first jet bundle $J_1(E)$ that parametrizes the first-order Taylor
expansions of the sections of $E$.  More precisely, let $\mathcal{J}$ be the
ideal sheaf defining the diagonal embedding $X \hookrightarrow X \times X$ and
let
$\operatorname{pr}_1, \operatorname{pr}_2 \colon
\operatorname{Zeroes}(\mathcal{J}^2) \to X$ be the restrictions of the
projections $X \times X \to X$ to the closed subscheme
$\operatorname{Zeroes}(\mathcal{J}^2) \subset X \times X$. The \emph{first jet
  bundle} is
$J_1(E) \coloneqq (\operatorname{pr}_1)_* \operatorname{pr}_2^* E$; this is
also called the bundle of principal parts in \cite{Fulton}*{Example~2.5.6}.
The vector bundle $J_1(E)$ has rank $n+1$ and sits in the short exact sequence
\[
  \begin{tikzcd}
    0 \ar[r] & \Omega_X  \otimes E \ar[r] & J_1(E) \ar[r] 
    & E \ar[r] & 0 \, .
  \end{tikzcd}
\]
The vector bundle $E$ is \emph{$1$-jet spanned} if the evaluation map
$V_E \otimes_{\CC} \OO_{X} \to J_1(E)$ is surjective; see
\cite{BDS}*{Subsection~1.3}.  With this concept, we have the following
corollary.
 
\begin{corollary}
  \label{c:1span}
  Let $X$ be an $n$-dimensional smooth projective variety and let $E$ be a
  very ample vector bundle of rank $e \leqslant n$.  Assuming that $E$ is
  $1$-jet spanned, the ramification locus $R(q_E)$ is reduced and
  Cohen-Macaulay of codimension $n-e+1$, so $\Delta(E) =
  (q_E)_*[R(q_E)]$. Furthermore, the discriminant locus $\Delta(E)$ is a
  hypersurface if and only if we have $\delta(E) >0$.  When $\Delta(E)$ is a
  hypersurface, the degree of discriminant locus is
  \[
    \deg \Delta(E) = \bigl\{\! c (p^* \Omega_X) \;
    s \bigl(p^* E^* \otimes \OO_{\PP(M^*)}(-1) \!\bigr) \!\bigr\}_{\!n-e+1}
    \, .
  \]
\end{corollary}

\begin{proof}
  By Theorem~\ref{t:ample} and Lemma~\ref{l:CM}, it suffices to show the
  existence of an injective vector bundle morphism from
  $E^* \otimes (\Omega_X)^* $ to $M_E^*$ or equivalently a surjective map from
  $M_E$ to $E \otimes \Omega_X$.  To establish this, we combine the defining
  short exact sequence for $M_E$ with the canonical short exact sequence for
  $J_1(E)$ to obtain the following commutative diagram with exact rows:
  \[
    \begin{tikzcd}[row sep = 1.0em]
      0 \ar[r] & M_E \ar[r] \ar[d] & V_C \otimes_{\CC} \OO_X \ar[r] \ar[d]
      & E \ar[r] \ar[equal]{d} & 0 \, \phantom{.} \\ 
      0 \ar[r] & \Omega_X \otimes E \ar[r] & J_1(E) \ar[r] 
      & E \ar[r] & 0 \, .
    \end{tikzcd}
  \]
  Since $E$ is $1$-jet spanned, the second vertical map is surjective.  Hence,
  the snake lemma implies that the first vertical map is also surjective. 
\end{proof}

\begin{remark}
  Remark~0.3.2 in \cite{BFS} establishes that, for any very ample line bundle
  $L$ on an $m$-dimensional smooth projective variety $Y$, we have
  $\defect(L) > 0$ if and only if $c_m \bigl( J_1(L) \!\bigr) = 0$.  When the
  discriminant locus $\Delta(L)$ is a hypersurface, this remark also shows
  that $\deg \Delta(L) = \int_Y c_m \bigl( J_1(L) \!\bigr)$.

  Given an $n$-dimensional smooth projective variety $X$ and a very ample
  vector bundle $E$ on $X$ of rank $e \leqslant n$, Lanteri and Mu\~{n}oz
  compute the top Chern class of the first jet bundle of the line bundle
  $L \coloneqq \OO_{\PP(E)}(1)$.  More precisely, when $Y = \PP(E)$,
  Proposition~1.1 in \cite{LM} expresses $c_{n+e-1} \!\bigl( J_1(L) \!\bigr)$
  as a polynomial in the Chern classes of $E$ and the tangent bundle $T_X$.
  Under the assumption that the vector bundle $E$ is $1$-jet spanned,
  Corollary~\ref{c:1span} provides a different formula for the degree of
  $\Delta(E)$.
\end{remark}


\begin{example}
  Let $L$ be a very ample line bundle on a smooth projective variety $X$. The
  line bundle $L$ is $1$-jet spanned; see \cite{BDS}*{Subsection~1.3}.  When
  the discriminant locus $\Delta(L)$ is a hypersurface, Example~\ref{e:e=1}
  and Corollary~\ref{c:1span} show that
  \[
    \deg \Delta(L) 
    = \sum_{i=0}^n (i+1) \, \int_X c_{n-i}(\Omega_X) \; c_1(L)^i
    \, .
  \]
  Therefore, we recover the degree of the classical discriminant; see
  \cite{GKZ}*{Example~3.12}.
\end{example}

Our second corollary focuses on vector bundles whose rank equals the dimension
of their underlying variety.  Part of this result provides an alternative
proof for Proposition~2.2 in \cite{LM}.

\begin{corollary}
  \label{c:n=e}
  Let $X$ be a $n$-dimensional smooth complex projective variety. For any very
  ample vector bundle $E$ of rank $n$ on $X$, the discriminant locus
  $\Delta(E)$ is irreducible and $\defect(E) > 0$ if and only if $X = \PP^n$
  and $E = \bigoplus_{i=1}^{n} \OO_{\PP^n}(1)$.  Assuming that $E$ is $1$-jet
  spanned and
  $(X,E) \neq \bigl( \PP^n, \bigoplus_{i=1}^n \OO_{\PP^n}(1) \bigr)$, the
  discriminant locus $\Delta (E)$ is an irreducible hypersurface of degree
  \[
    \int_X \bigl( c_1(\Omega_X) + c_1(E) \!\bigr) \, c_{n-1}(E) + n \, c_n(E)
    \, .
  \]
\end{corollary}

\begin{proof}
  Theorem~\ref{t:ample} and Example~\ref{e:n=e} show that $\Delta(E)$ is
  irreducible and $\defect(E) > 0$ if and only if
  \[
    \delta(E)
    = \int_X \bigl( c_1(\Omega_X) + c_1(E) \!\bigr) \, c_{n-1}(E) + n \, c_n(E)
    = 0 \, .
  \]
  When $(X, E) = \bigl( \PP^n, \bigoplus_{i=1}^n \OO_{\PP^n}(1) \bigr)$, we
  have $\delta(E) = \bigl(\! (-n-1) + n \bigr) n + n = 0$ and
  $\defect(E) > 0$.  Hence, it suffices to show that, for any very ample $E$
  excluding $\bigoplus_{i=1}^n \OO_{\PP^n}(1)$, we have $\delta(E) > 0$.  If
  $E$ is $1$-jet spanned as well as very ample, then Corollary~\ref{c:1span}
  shows that $\deg \Delta(E) = \delta(E)$.

  Since $E$ is very ample, we have $\int_X c_n(E) >0$; see
  \cite{BG}*{Proposition~2.2}. Thus, it is enough to prove that
  $\int_X \bigl( c_1(\Omega_X) + c_1(E) \!\bigl) \, c_{n-1}(E) \geqslant 0$.
  Let $K_X$ be the canonical divisor on $X$ and let $D$ be the Cartier divisor
  associated to $\det(E)$. Since $E$ is very ample, $D$ is also.  Moreover,
  Theorem~2 in \cite{YZ} establishes that the adjoint divisor $K_X + D$ is nef
  unless $(X ,E) = \bigl( \PP^n, \bigoplus_{i=1}^{n} \OO_{\PP^n}(1)
  \bigr)$. The very ampleness of the vector bundle $E$ implies that
  $c_{n-1}(E) \neq 0$; again see~\cite{BG}*{Proposition~2.2}.  We deduce that
  $c_{n-1}(E)$ is the class of a curve $C$ by a Bertini-type argument; see
  \cite{Ionescu}*{Theorem~B}. It follows that
  \[
    \int_X \bigl( c_1(\Omega_X) + c_1(E) \!\bigr) c_{n-1}(E)
    = (K_X + D) \cdot C
    \geqslant 0 \, . \qedhere
  \]
\end{proof}

To illustrate this corollary, we quickly recompute the degree of the
discriminant locus for nonnegative twists of the tangent bundle on
$\PP^n$; see \cite{ASS}*{Corollary~4.2} and \cite{A}*{Example~4.9}.

\begin{example}
  Let $d$ be a nonnegative integer and let $T_{\PP^n}$ be the tangent bundle
  on $\PP^n$.  We have
  $c_1(\Omega_{\PP^n}) = c_1\bigl( \OO_{\PP^n}(-n-1) \!\bigr)$. From
  the Euler sequence
  \[
    \begin{tikzcd}
      0 \arrow[r]
      & \OO_{\PP^n} \arrow[r]
      & \smash{\displaystyle\bigoplus\limits_{i=1}^{n}} \,\, \OO_{\PP^n}(1) \arrow[r]
      & T_{\PP^n} \arrow[r] & 0 \, ,
    \end{tikzcd}
  \]
  we deduce that
  \[
    \int_{\PP^n} c_i \bigl( T_{\PP^n}(d) \!\bigr)
    = \sum_{j=0}^i \binomial{n-j}{i-j} d^{i-j} \binomial{n+1}{j}
  \]
  for all nonnegative integers $i$.  Combining Propositions~2.1--2.3 in
  \cite{BDS}, the Euler sequence also shows that vector bundle $T_{\PP^n}(d)$
  is very ample and $1$-jet spanned. Thus, Corollary~\ref{c:n=e} establishes
  that the discriminant locus $\Delta \bigl( T_{\PP^n}(d) \!\bigr)$ is an
  irreducible hypersurface and
  \begin{align*}
    \deg \Delta \bigl( T_{\PP^n}(d) \!\bigr)
    &= n d \sum_{j=0}^{n-1} (n-j) \, d^{n-1-j} \binomial{n+1}{j}
      + n \sum_{j=0}^n d^{n-j} \binomial{n+1}{j} \\[-2pt]
    &= n \sum_{j=0}^{n} d^{n-j} (n+1-j) 
      \binomial{n+1}{n+1-j} \\[-6pt]
    &= n(n+1) \sum_{j=0}^n d^{n-j} \binomial{n}{j} 
      = n(n+1)(d+1)^n
      \, .  \qedhere
  \end{align*}
\end{example}

\section{Bogomolov Instability Theorem}

\noindent
In this section, we use calculations involving the discriminant divisor of a
multi-section to give a simple proof of the Bogomolov instability theorem for
vector bundles having rank $2$ on an algebraic surface. At the very least, it
was known to experts that one could give an argument along these lines.
However, since it fits well with the themes of this note and is not widely
known, we felt it worthwhile to include it here.  We refer the reader to
\cite{Langer} for another approach having several points of contact with the
present proof.

Let $X$ be a smooth complex projective surface. We consider a vector bundle
$E$ of rank $2$ on $X$, and denote by $D$ a Cartier divisor associated to
$\det(E)$.  The vector bundle $E$ is \emph{Bogomolov unstable} if there exist
a divisor $A$ and a finite scheme $W \subset X$ (possibly empty) such that the
sequence
\[
  \begin{tikzcd}
    0 \ar[r]
    & \OO_{X}(A) \ar[r]
    & E \ar[r]
    & \OO_{X}(D-A) \otimes \mathcal{I}_W \ar[r]
    & 0 \, ,
  \end{tikzcd}
\]
is exact, $(2A - D)^2 > 4 \, \operatorname{length}(W)$, and
$(2A-D) \cdot H > 0$ for some (or any) ample divisor $H$ on $X$.  Roughly
speaking, being Bogomolov unstable means that the vector bundle $E$ contains
an unexpectedly positive subsheaf.

Bogomolov's theorem asserts that instability is detected numerically.

\begin{theorem}
  \label{t:bog}
  The vector bundle $E$ is Bogomolov unstable if and only if
  \[
    \int _X c_1(E)^2 - 4 \, c_2(E) > 0 \, .
  \]
\end{theorem}

\noindent
The defining exact sequence for a Bogomolov unstable vector bundle implies
that 
\[
  \int_X c_2(E) = \operatorname{length}(W) + A \cdot (D-A) \, ,
\]
so the inequality holds. Thus, the essential content of the theorem is the
converse statement: the inequality implies the existence of a destabilizing
subsheaf $\OO_X(A)$.

For our proof of this implication, suppose that
$\int _X \bigl( c_1(E)^2 - 4 \, c_2(E) \!\bigr) > 0$.  Let
$\pi \colon \PP(E) \to X$ the projectivization of $E$, so $\dim \PP(E) =
3$. The starting point, as in other arguments, is the next lemma.

\begin{lemma}
  When the vector bundle $E$ satisfies the inequality in Theorem~\ref{t:bog},
  the line bundle $\OO_{\PP(E)}(2) \otimes \pi^*\OO_X(-D)$ on $\PP(E)$ is
  big. In other words, there is a positive number $C > 0$ such that, for all
  sufficiently large integers $m$, we have
  \[
    h^{0} \bigl(\PP(E), \OO_{\PP(E)}(2m) \otimes \pi^*\OO_X(-mD) \!\bigr)
    = h^{0} \bigl( X, \Sym^{2m}(E) \otimes \OO_X(-mD) \!\bigr)
    \geqslant C \, m^3 \, .
  \]
\end{lemma}

\begin{proof}[Idea of proof]
  The asymptotic Riemann--Roch theorem~\cite{Lazarsfeld}*{Theorem~1.1.24}
  shows that
  \[
    \chi \bigl( X, \Sym^{2m}(E) \otimes \OO_X(-mD) \!\bigr)
    = \tfrac{1}{3} \bigl( c_1^2(E) - 4 \, c_2(E) \!\bigr) \, m^3 + O(m^2) \, .
  \]
  The assertion follows via Serre duality and the fact that the vector bundle
  $\Sym^{2m}(E) \otimes \OO_X(-mD)$ has trivial determinant.  For more
  details, see \cite{Reid}*{Proposition~2}.
\end{proof}

Now let $H$ be an ample divisor on $X$. By an argument of
Kodaira~\cite{Lazarsfeld}*{Proposition~2.2.6}, it follows from the lemma that,
for all sufficiently large integers $m$, we have
\[
  H^{0} \bigl( \PP(E), \OO_{\PP(E)}(2m) \otimes \pi^* \OO_{X}(-mD - H) \!\bigr)
  \neq 0 \, .
\]
Fix one such integer $m$ and choose nonzero section
$s \in H^{0}\bigl( \PP(E), \OO_{\PP(E)}(2m) \otimes \pi^* \OO_X(-mD - H)
\bigr)$.  Let $Z \coloneqq \operatorname{Zeroes}(s)$ be the zero locus of the
global section $s$. Hence, the subscheme $Z$ is a divisor on $\PP(E)$ of
relative degree $2m$ over $X$.

We study the irreducible components of $Z$ with the aim of singling out a
particularly interesting one. To begin, let $Z_0 \subset \PP(E)$ denote the
union of any ``vertical'' components of $Z$: $Z_0$ is the preimage under $\pi$
of the zeroes of a section of $\OO_X(-A_0)$ for some anti-effective divisor
$A_0$ on $X$.  Write $Z_1, Z_2, \dotsc, Z_t \subset \PP(E)$ for the remaining
irreducible components of $Z$ allowing repetitions to account for
multiplicities.  In other words, each $Z_i \subset \PP(E)$ is a reduced and
irreducible divisor that is defined by a section of
$\OO_{\PP(E)}(d_i) \otimes \pi^* \OO_X( - A_i)$ for some divisor $A_i$ on $X$
and positive integer $d_i$. By construction, the divisor
$A_0 + A_1 + \dotsb + A_t$ is linearly equivalent to $mD + H$ and
$d_1 + d_2 + \dotsb + d_t = 2m$, so the divisor
$ \sum_{i \geqslant 1} \bigl( A_i - \tfrac{d_i}{2} D \bigr)$ is numerically
equivalent to $H - A_0$.  Since $-A_0$ is an effective divisor, it follows
that
$\bigl( \textstyle\sum_{i \geqslant 1} 2 \, A_i - d_i \, D \bigr) \cdot H >
0$.  By reindexing the components if necessary, we may assume that
$(2 \, A_1 - d_1 \, D) \cdot H > 0$.

The idea is to consider the discriminant divisor $\Delta \subseteq X$ over
which the fibre of the map $Z_1 \to X$ is not $d_1$ distinct points.
Specifically, Proposition~\ref{p:branch} shows that the class of $\Delta$ is
given by
\[
  \delta = d_1(d_1-1) \, D - 2(d_1-1) \, A_1 
\]
and $\delta$ is either effective or zero, so $\delta \cdot H \geqslant
0$. However, if $d_1 > 1$, then this would contradict the assumption that
$(2 \, A_1 - d_1 \, D) \cdot H > 0$.  Therefore, we have $d_1 = 1$ and $Z_1$
is defined by a (necessarily saturated) section in
$H^{0}\bigl( \PP(E), \OO_{\PP(E)}(1) \otimes \pi^* \OO_X(-A_1) \!\bigr)$. The
corresponding section in $H^{0} \!\bigl( X, E \otimes \OO_X(-A) \!\bigr)$
defines a closed subscheme $W$ of $X$ and gives rise to a short exact sequence
\[
  \begin{tikzcd}
    0 \ar[r]
    & \OO_{X}(A_1) \ar[r]
    & E \ar[r]
    & \OO_{X}(D-A_1) \otimes \mathcal{I}_W \ar[r]
    & 0 \, .
  \end{tikzcd}
\]
The Bogomolov inequality $\int _X c_1(E)^2 - 4 \, c_2(E) > 0$ implies that
$(2A-D)^2 > 4 \, \operatorname{length}(W)$ and $(2A-D) \cdot H > 0$.
Therefore, we have established that the vector bundle $E$ is unstable.

It remains to prove the following proposition.

\begin{proposition}
  \label{p:branch}
  Let $E$ be a rank two vector bundle on $X$ with $\det(E) = \OO_X(D)$, let
  $\pi \colon \PP(E) \to X$ be the projectivization of $E$, and consider a
  reduced and irreducible divisor
  \[
    \begin{tikzcd}[row sep = 1.0em]
      Y \ar[rr, hook] \ar[dr, "f\;\;\;" below] & & \PP(E) \ar[dl, "\pi"] \\
      & X &
    \end{tikzcd}
  \]
  defined by a section of $\OO_{\PP(E)}(d) \otimes \pi^*\OO_X(-A)$ for some
  positive integer $d$. The locus $\Delta(f) \subseteq X$ of points $x \in X$
  over which the fibre $f^{-1}(x)$ fails to consist of $d$ distinct points
  supports an effective divisor in the class
  $\delta = d(d-1) \, D + 2(d-1) \, A$.  In particular, this class is
  effective or zero.
\end{proposition}

\begin{proof}
  Consider the set
  $\Gamma \coloneqq \bigl\{ y \in Y \mathrel{\big|} \text{ $f$ is not \'etale
    at $y$} \bigr\}$.  The map $f$ is generically \'etale because $Y$ is
  reduced. It follows that $\Gamma$ has dimension $1$ (or is empty) and
  $\Delta(f) = f(\Gamma)$. We claim that, viewed as a cycle of codimension $2$
  on $\PP(E)$, $\Gamma$ supports the effective class
  \begin{equation}
    \label{d:class}
    \tag{$*$}
    \gamma \coloneqq \Bigl( \! (d-2) \, c_1 \big( \OO_{\PP(E)}(1) \!\big) +
    \pi^* (D-A) \!\Bigr) \cdot \Bigl( d \, c_1 \big( \OO_{\PP(E)}(1) \!\big) -
    \pi^* A \!\Bigr) \, .
  \end{equation}
  There are at least two ways to confirm this claim.
  \begin{compactitem}[$\bullet$]
  \item As a cycle on $Y$, $\gamma$ is the class of the first degeneracy locus of
    the induced differential $df \colon f^* \Omega_X \to \Omega_Y$, so
    $\gamma = c_1(\Omega_Y - f^* \Omega_X) = c_1 \bigl( \OO_Y(K_Y - f^* K_X)
    \!\bigr)$ which is the class of the relative canonical divisor
    $K_{Y/X} \coloneqq K_Y - f^* K_X$.  By the adjunction formula, we have
    $K_{Y/X} = (K_{\PP(E)/X} + Y)|_{Y}$.  Thus, as a cycle on $\PP(E)$, we
    have $\gamma = [(K_{\PP(E)/X} + Y)|_{Y}] \cap [Y]$.  Since
    $[Y] = c_1\bigl( \OO_{\PP(E)}(d) \otimes \pi^* \OO_X(-A) \!\bigr)$, the
    equation \eqref{d:class} follows from
    $[K_{\PP(E)/X}] = c_1\bigl( \OO_{\PP(E)}(-2) \otimes \pi^* \OO_X(D)
    \!\bigr)$; see \cite{Lazarsfeld}*{Section~7.3.A}.
  \item The section $s$ in
    $H^{0} \bigl( X, \OO_{\PP(E)}(d) \otimes \pi^*\OO_X(-A) \!\bigr)$ defining
    $Y$ lifts to a section of the first relative jet bundle of
    $ds \in H^{0} \!\left( \PP(E), J_1^{\pi} \!\bigl( \OO_{\PP(E)}(d) \otimes
      \pi^*\OO_X(-A) \!\bigr) \!\right)$, and
    $\Gamma = \operatorname{Zeroes}(ds)$. From the canonical short exact
    sequence
    \[
      \begin{tikzcd}[column sep = 1.0em]
        0 \ar[r]
        & \Omega_{\PP(E)/X}(d) \otimes \pi^*\OO_X(-A) \ar[r]
        & J_1^{\pi} \!\bigl( \OO_{\PP(E)}(d) \otimes \pi^*\OO_X(-A) \!\bigr) \ar[r]
        & \OO_{\PP(E)}(d) \otimes \pi^*\OO_X(-A) \ar[r]
        & 0
      \end{tikzcd}
    \]
    we see that
    $\gamma = c_2 \left( J_1^{\pi} \!\bigl( \OO_{\PP(E)}(d) \otimes
      \pi^*\OO_X(-A) \!\bigr) \!\right)$, which again establishes the equation
    \eqref{d:class}.
  \end{compactitem}
  It remains to check that $\pi_*(\gamma) = \delta$. This follows from the
  Grothendieck relation
  \[
    c_1 \bigl( \OO_{\PP(E)}(1) \! \bigr)^2
    - \pi^* \bigl( c_1(E) \!\bigr) \cdot c_1 \bigl( \OO_{\PP(E)}(1) \! \bigr)
    + \pi^* \bigl( c_2(E) \!\bigr) = 0 \, ,
  \]
  $\pi_* \!\Bigl(\! \pi^* (\alpha) \cdot c_1 \bigl( \OO_{\PP(E)}(1) \! \bigr)
  \!\!\Bigr) = \alpha$, and $\pi_* \big( \pi^*(\beta) \big) = 0$ for any
  classes $\alpha \in A^1(X)$ and $\beta \in A^2(X)$.
\end{proof}

\subsection*{Acknowledgements}

We thank Antonio Lanteri and Roberto Mu\~noz for useful insights in the
history of discriminant loci.  We are also grateful to Adrian Langer for
drawing our attention to \cite{Langer} and for some valuable discussions.  The
second author was partially supported by NSF grant DMS-1739285 and the third
author was partially supported by NSERC.


\raggedright

\end{document}